\documentclass[11pt, reqno]{amsart}

\usepackage{amssymb,latexsym,amsmath,amsfonts,amsthm}
\usepackage{mathrsfs}
\usepackage{enumitem}
\usepackage{soul}
\usepackage[usenames]{color}
\usepackage{hyperref}
\usepackage{comment}

\allowdisplaybreaks

\numberwithin{equation}{section}

\definecolor{DPurple}{rgb}{0.46,0.2,0.69}

\theoremstyle{definition}
\newtheorem{definition}{Definition}[section]

\theoremstyle{remark}
\newtheorem{remark}[definition]{Remark}

\theoremstyle{plain}
\newtheorem{theorem}[definition]{Theorem}
\newtheorem{result}[definition]{Result}
\newtheorem{lemma}[definition]{Lemma}
\newtheorem{proposition}[definition]{Proposition}

\newtheorem{corollary}[definition]{Corollary}

\newcommand{\eps}{\varepsilon}
\newcommand{\zt}{\zeta}
\newcommand{\zbar}{\overline{z}}

\newcommand{\cHess}{\mathfrak{H}_{\raisebox{-2pt}{$\scriptstyle {\mathbb{C}}$}}}

\newcommand{\dbar}{\overline\partial}
\newcommand{\opCMA}{(dd^c(\bcdot))^n}

\newcommand{\bdy}{\partial}

\newcommand{\cont}{\mathcal{C}(\overline{\Omega})}
\newcommand{\smoo}{\mathcal{C}}

\newcommand\leb[1]{\mathbb{L}^{{#1}}}
\newcommand{\omgclass}{\mathcal{C}_{\omega}(\overline{\Omega})}
\newcommand{\Bloclass}{\widehat{\mathscr{B}}(\phi, f)}

\newcommand{\bcdot}{\boldsymbol{\cdot}}

\newcommand{\btl}{\blacktriangleleft}

\newcommand{\V}{\mathcal{V}}

\newcommand{\Omegabar}{\overline{\Omega}}
\newcommand{\omegatau}{\omega(\|\tau\|)}

\newcommand{\Z}{\mathbb{Z}}
\newcommand{\N}{\mathbb{N}}

\newcommand{\Cn}{\mathbb{C}^n}

\newcommand{\C}{\mathbb{C}} 
\newcommand{\R}{\mathbb{R}}

\begin{document}

\title[Regularity of solutions to the complex Monge--Amp{\`e}re
equation]{Regularity of solutions to the complex Monge--Amp{\`e}re
equation with a general modulus of continuity}

\author{Annapurna Banik}
\address{Tata Institute of Fundamental Research,
Centre for Applicable Mathematics, Bengaluru 560065, India}
\email{banikannapurna01@gmail.com, annapurna25@tifrbng.res.in}

\begin{abstract}
We investigate the regularity theory for the complex Monge--Amp{\`e}re equation
when, owing to the data or to the underlying domain, 
its solution fails to be in any H{\"o}lder class and to have a general concave
modulus of continuity. Our first main result is motivated by, and generalizes,
a result of S.-Y.~Li. As an application, we prove a result on the regularity of solutions to the
complex Monge--Amp{\`e}re equation on general $B$-regular domains in $\mathbb{C}^n$.
\end{abstract}


\maketitle

\vspace{-5mm}

\section{Introduction and statement of main results}\label{sec:intro}

Let $\Omega \varsubsetneq \Cn$, $n\geq 2$, be a bounded
domain. Given two functions $\phi \in \smoo(\bdy\Omega; \R)$ and 
$f\in \smoo(\Omega)$, $f\geq 0$, the \emph{Dirichlet problem for the complex
Monge--Amp{\`e}re equation} is the non-linear boundary-value problem that seeks
a function $u \in \smoo(\overline\Omega; \R)$ such that
\begin{align} \label{eqn:D-P_CMA}
  (dd^c{u})^n\,
  :=\,\underbrace{dd^c{u}\wedge\dots \wedge dd^c{u}}_{n \ \text{factors}}\,
  &=\,f\V_n, \ \ u \in  \cont \cap {\rm psh}(\Omega), \\
  u|_{\bdy\Omega}\,&=\,\phi. \notag
\end{align}
Here, $d = (\bdy + \dbar)$ is the exterior derivative, $d^c := i(\bdy - \dbar)$, and
$\V_n$ is defined as
$$
  \V_n := (i/2)^n(dz_1\wedge d\zbar_1)\wedge\dots
  \wedge (dz_n\wedge d\zbar_n).
$$
When $u|_{\Omega}\notin \smoo^2(\Omega; \R)$, the left-hand side of
\eqref{eqn:D-P_CMA} must be interpreted as a current of bidegree $(n, n)$.
That this makes sense when $u\in {\rm psh}(\Omega)\cap \smoo(\Omega)$ was established by
Bedford--Taylor in their seminal paper \cite{bedford-taylor:1976}, where they established
an existence and uniqueness theorem
for \eqref{eqn:D-P_CMA} with continuous data $(f, \phi)$ for strongly pseudoconvex domains.
Later, B\l{}ocki \cite{blocki:1996} showed that such a theorem holds true in a much larger class of
domains, namely, $B$-regular domains (see Section~\ref{subsec:defn-prpn} for the definition and
further discussion). There is extensive literature on the complex Monge--Amp{\`e}re equation
(too numerous to discuss here) in complex analysis, and also in complex geometry;
the reader is referred to \cite{bedford-taylor:1976, caffarelli-kohn-nirenberg-spruck:1985, caffarelli-nirenberg-spruck:1985, kolodziej:1998,
guedj-zeriahi:2017} and the references therein.
\smallskip

The regularity theory for the complex Monge--Amp{\`e}re equation is of great interest.
Loosely speaking, this concerns determining\,---\,assuming $\bdy \Omega$
``nice'' enough\,---\,when the solution to \eqref{eqn:D-P_CMA} has some regularity
(e.g., H{\"o}lder continuity), {\em for all} sufficiently regular data $(f, \phi)$.
The first such result was proved by Bedford--Taylor \cite{bedford-taylor:1976} for $\Omega$
strongly pseudoconvex. We refer the reader to \cite{coman:1997, li:2004,
guedj-kolodziej-zeriahi:2008,
charabati:2015, ha-khanh:2015, baracco-khanh-pinton-zampieri:2016,
nguyen:2018, nguyen:2020} for further work on this subject.
There are many applications of the complex Monge--Amp{\`e}re equation
in complex geometry; see the survey \cite{phong-song-sturm:2012}. There have been
relatively fewer applications of the complex Monge--Amp{\`e}re equation on domains
$\Omega\varsubsetneq \Cn$ outside the realm of PDEs (for $\Omega$ not the unit ball in
$\C^n$). This is largely due to the need for regularity results for solutions to
\eqref{eqn:D-P_CMA}\,---\,see \cite{li:2004, banik-bharali:2024, bharali-masanta:2025} for some
applications to domains in $\Cn$. 
\smallskip

An interesting advance in this subject was made by S.-Y.~Li \cite{li:2004}, who proved a
H{\"o}lder-regularity result, in the spirit of the one by Bedford--Taylor
\cite{bedford-taylor:1976}, for \eqref{eqn:D-P_CMA} on weakly pseudoconvex domains.
He studied the Dirichlet problem \eqref{eqn:D-P_CMA} on domains that admit a
uniformly strictly plurisubharmonic function $\rho \in \cont \cap {\rm psh}(\Omega)$
(possibly non-smooth), which is the basis for a sufficient condition for the H{\"o}lder regularity
of solutions to \eqref{eqn:D-P_CMA}. In this paper, we are motivated in particular by the result of S.-Y. Li stated below.

\begin{result}[Li, {\cite[Corollary~2.2]{li:2004}}] \label{res:Li-Holder-reg}
Let $\Omega \varsubsetneq \Cn$ be a bounded pseudoconvex domain with $\smoo^2$-smooth
boundary. Assume that there exist
$\varepsilon \in (0,1]$ and a function $\rho \in \smoo^{0,\,\varepsilon}
(\overline{\Omega}) \cap {\rm psh}(\Omega)$ such that $ \Omega= {\rho}^{-1}(-\infty, 0)$,
$\bdy \Omega = \rho^{-1}\{0\}$, and $\rho - \|\bcdot\|^2$ $ \in {\rm psh}(\Omega)$.
Then, $\Omega$ admits a unique solution $u$ to \eqref{eqn:D-P_CMA} of class
$\smoo^{0, \eps}(\overline{\Omega})$ for any $\phi \in \smoo^{1,1}(\bdy \Omega; \R)$ and
for any $f \geq 0$ satisfying $f^{1/n} \in \smoo^{0, \eps}(\overline{\Omega})$.
\end{result}

Our first main result gives a generalization of
Result~\ref{res:Li-Holder-reg}; see Theorem~\ref{th:general-reg-CMA} below. But
first, we present a corollary to Result~\ref{res:Li-Holder-reg} which says that the
sufficient condition for the H{\"o}lder regularity of solutions to \eqref{eqn:D-P_CMA} given by
Result~\ref{res:Li-Holder-reg} is also necessary. This observation would have
been just a remark \textbf{but} for the fact that Corollary~\ref{cor:Holder-iff} makes no assumptions on
the smoothness of $\bdy\Omega$.  

\begin{corollary} \label{cor:Holder-iff}
Let $\Omega \varsubsetneq \Cn$ be a bounded domain.
The following are equivalent:
\begin{itemize}
    \item[$(a)$] There exist $\varepsilon \in (0,1]$ 
    and a function $\rho \in \smoo^{0,\,\varepsilon}(\overline{\Omega}) \cap {\rm psh}(\Omega)$
    such that $ \Omega= {\rho}^{-1}(-\infty, 0)$, $\bdy \Omega = \rho^{-1}\{0\}$,
    and $\rho - \|\bcdot\|^2$ $ \in {\rm psh}(\Omega)$.

    \smallskip
    
    \item[$(b)$] There exists $\varepsilon \in (0,1]$ such that $\Omega$
    admits a (unique) solution $u$ to the Dirichlet problem
    \eqref{eqn:D-P_CMA} of class $\smoo^{0, \varepsilon}(\overline{\Omega})$,
    for any $\phi \in \smoo^{1, 1}(\bdy \Omega; \R)$
    and for any $f$ such that either $f\equiv 0$ or $f > 0$ with
    $f \in \smoo^{0,\,\varepsilon}(\overline{\Omega})$.
\end{itemize}
\end{corollary}

\begin{remark}
Despite the above result being formulated for general bounded domains,
we view this as a corollary to Result~\ref{res:Li-Holder-reg}. This is because 
a careful reading the {\em proof} of Result~\ref{res:Li-Holder-reg} in \cite{li:2004} suggests
that no assumptions on the regularity of $\bdy \Omega$ are required. That this is true is
validated by the proof of Theorem~\ref{th:general-reg-CMA}. Also note that $\Omega$ is not
assumed to be pseudoconvex \emph{a priori}; see the discussion following 
Theorem~\ref{th:general-reg-CMA}.
\end{remark}

\begin{remark} \label{rem:density-data}
We can establish a more general result than Corollary~\ref{cor:Holder-iff} by
relaxing the condition on $f$, but doing so results in a statement involving
$f^{1/n}$ that is more technical, while we prefer the more aesthetic result.
Moreover, we are motivated by the applications of the regularity theory for
the complex Monge--Amp{\`e}re equation outside the realm of PDEs, where the
conditions on $f$ as stated in Corollary~\ref{cor:Holder-iff} are more frequently
encountered (e.g., $f \equiv 0$ relates to the Perron--Bremermann envelope).
At any rate, the case when $f \geq 0$ will be discussed in connection with
the more general result alluded to above; see Theorem~\ref{th:general-reg-CMA} (and also Theorem~\mbox{\ref{th:B-reg-CMA}}).
\end{remark} 
 
In the previous result and in what follows,
$\phi \in \smoo^{1, 1}(\bdy \Omega)$ means that it is the restriction to $\bdy \Omega$ of a
$\mathcal{C}^{1,1}$-smooth function defined on an open neighbourhood of $\bdy \Omega$.
\smallskip

We now state our first main result. Note that, given a domain $\Omega
\varsubsetneq \Cn$ and the corresponding H{\"o}lder exponent $\eps \in (0, 1]$
as featured in Result~\ref{res:Li-Holder-reg}, the power function
$r \mapsto Cr^\eps$, $r\geq 0, \ C>0$,
serves as a modulus of continuity for the functions in the hypothesis and
the conclusion of Result~\ref{res:Li-Holder-reg}. In this paper,
we wish to allow domains $\Omega$ admitting an exhaustion function $\rho$
analogous to that in Result~\ref{res:Li-Holder-reg} but whose modulus of continuity is
``worse'' than\,---\,i.e., not comparable to\,---\,any power and investigate the regularity of
solutions to \eqref{eqn:D-P_CMA}.
(See Definition~\ref{defn:mod-cont} for the precise
definition of modulus of continuity). With those words, we state:

\begin{theorem}\label{th:general-reg-CMA}
Let $\Omega \varsubsetneq \Cn$ be a bounded domain. Assume that there exists a (concave)
modulus of
continuity $\omega$ and a function $\rho \in \omgclass \cap {\rm psh}(\Omega)$
such that $ \Omega= {\rho}^{-1}(-\infty, 0)$,
$\bdy \Omega = \rho^{-1}\{0\}$, and $\rho -  \|\bcdot\|^2 \in {\rm psh}(\Omega)$.
Then, $\Omega$ admits a solution $u$ to \eqref{eqn:D-P_CMA} of class
$\smoo_{\omega}(\overline{\Omega})$ for any $\phi \in \smoo^{1,1}(\bdy \Omega; \R)$ and
for any $f \geq 0$ satisfying $f^{1/n} \in \smoo_{\omega}(\overline{\Omega})$.
\end{theorem}

The reader is referred to \eqref{eqn:omega-class} for the definition of $\omgclass$.
A few lines comparing Result~\ref{res:Li-Holder-reg}
and Theorem~\ref{th:general-reg-CMA} are in order. To reiterate, the proof of
Result~\ref{res:Li-Holder-reg} does not require the assumption that
$\bdy \Omega$ be $\smoo^2$-smooth. {Apart from the wider
range of domains that our assumption on $\rho$ allows, not assuming any smoothness of 
$\bdy\Omega$ too can be useful. See, for instance,} our next main result,
Theorem~\ref{th:B-reg-CMA}, whose proof relies on Theorem~\ref{th:general-reg-CMA}.
Note also that our assumptions on the function $\rho$ implies that $\Omega$ is pseudoconvex. This is why, we do not assume pseudoconvexity of $\Omega$
in the statement of Theorem~\ref{th:general-reg-CMA}. But the main novelty of
Theorem~\ref{th:general-reg-CMA} is that the $\rho\in \smoo_{\omega}(\overline{\Omega})$
condition allows us to address a much larger class of domains, on which the solutions to the problem \eqref{eqn:D-P_CMA} need not be H{\"o}lder continuous for all admissible data $(f,\phi)$
(see the last paragraph of this section). While Theorem~\ref{th:general-reg-CMA}
yields a conclusion analogous to that of
Result~\ref{res:Li-Holder-reg}, the weaker hypothesis of the former necessitates a proof that, at
some point, diverges markedly from that of Result~\ref{res:Li-Holder-reg}. An essential discussion on this matter is deferred to Remark~\ref{rem:detailed-proof}.

\begin{remark} \label{rem:bdy-value-data}
The choice for the regularity of the boundary data $\phi$ in Theorem~\ref{th:general-reg-CMA}
is natural. Firstly, one sees from the proofs in \cite{li:2004, guedj-kolodziej-zeriahi:2008},
that handling a class of boundary data that are strictly less regular
than $\smoo^{1,1}(\bdy \Omega)$ gets to be quite technical. If we allow an arbitrary modulus of
continuity $\omega$ for all the relevant functions, identifying the ``right'' class for
boundary data $\phi$ that are less regular than $\smoo^{1,1}(\bdy \Omega)$ will require a much
more protracted arguments (this issue will be a part of forthcoming work involving new
methods). Secondly, we are motivated (as stated earlier)
by consequenses of the regularity theory outside the realm of PDEs, where, in most applications, one usually works with
$\phi \in \smoo^{\infty}(\bdy \Omega)$\,---\,see 
\cite{li:2004, phong-song-sturm:2012, banik-bharali:2024, bharali-masanta:2025},
for instance.
\end{remark}

In addition to generalizing Result~\ref{res:Li-Holder-reg},
Theorem~\ref{th:general-reg-CMA} has an amusing consequence, which is our second main
result. Suppose $\Omega \varsubsetneq \Cn$ is a $B$-regular domain. As mentioned earlier in this
section, $\Omega$ admits a unique solution to \eqref{eqn:D-P_CMA} for all
admissible $(f, \phi)$; see Result~\ref{res:B-reg_CMA-soln}. However, in general,
it was not known whether these solutions are H{\"o}lder continuous, even when
$f, \ \phi$ are very regular. Recently, Bharali--Masanta \cite{bharali-masanta:2025}
provided a negative result in this direction. See Section~\ref{sec:two-proofs}
for more details, which will also suggest that it is meaningful to ask:
{\em is there a (concave) modulus of continuity $\omega$\,---\,that depends only on
$\Omega$\,---\,such that \eqref{eqn:D-P_CMA} has a solution in $\omgclass \cap {\rm psh}(\Omega)$ {\bf for all} admissible data $(f, \phi)$?} 
Theorem~\ref{th:general-reg-CMA} addresses this question and gives an affirmative answer.
To the best of our knowledge,
no such regularity result for solutions to the complex Monge--Amp{\`e}re equation
on {\em general} $B$-regular domains is known in the literature.
\begin{theorem} \label{th:B-reg-CMA}
Let $\Omega \varsubsetneq \Cn$ be a $B$-regular domain. Then, there exists a modulus of continuity
$\omega$ such that $\Omega$ admits a solution $u$ to \eqref{eqn:D-P_CMA}
of class $\smoo_{\omega}(\overline{\Omega})$ for any $\phi \in \smoo^{1,1}(\bdy \Omega; \R)$
and for any $f \geq 0$ satisfying
$f^{1/n} \in \smoo_{\omega}(\overline{\Omega})$.
\end{theorem}

\section{Analytical preliminaries}\label{sec:premln}

In this section, we note some important definitions and facts that are relevant to
our proofs.
\smallskip

\subsection{Common notations} \label{subsec:notations}
We write down a list of recurring notations (some of which were used without
comment in the previous section). In this list, $\Omega$ denotes a domain in $\Cn$.
\begin{enumerate}
  \item ${\rm psh}(\Omega)$ denotes the class of all plurisubharmonic functions
  defined on $\Omega$.
  \smallskip
  
  \item Given a $\smoo^2$-smooth function $\psi:\Omega \to \C$, $(\cHess{\psi})(z)$ is the
  complex Hessian of $\psi$ at $z \in \Omega$.%
  \smallskip
  
  \item For $v \in \R^d$, $\|v\|$ denotes the Euclidean norm. Given a non-empty set
  $S \subset \R^d$, ${\rm diam}(S): = \sup \{\|a-b\|: a, b \in S\}$.
  \smallskip
  
  \item Let $S (\neq \emptyset) \subset \Cn$. For fixed $j \in \N$, $\beta \in (0,1]$,
  $\smoo^{j,\,\beta}(S)$ is the class of all functions that are continuously
  differentiable to order $j$ and whose $j$-th partial derivatives satisfy a uniform H{\"o}lder
  condition with exponent $\beta$. $\smoo(S)$ denotes the class of continuous functions.
  The subclasses $\smoo^{j,\,\beta}(S; \R) \subset \smoo^{j,\,\beta}(S)$ and
  $\smoo(S; \R) \subset \smoo(S)$ consist of $\R$-valued functions.
\end{enumerate}

\subsection{The complex Monge--Amp{\`e}re operator}
We briefly recall some basic facts about the complex Monge--Amp{\`e}re operator
$\opCMA$, as featured in \eqref{eqn:D-P_CMA}, which will be relevant to the proof of
Theorem~\ref{th:general-reg-CMA}. Let $\Omega \subset \Cn$ be a domain.
Let $\mathcal{M}_{n}(\Omega)$ denote the class of currents on $\Omega$
of bidegree $(n, n)$ and of order $0$, equipped with  the usual topology
of weak convergence. It is well known that elements of $\mathcal{M}_{n}(\Omega)$
are identified with $(n,n)$-differential forms with measure coefficients.
\smallskip

Firstly, an easy computation yields, for any
$\smoo^2$-smooth function $U$ on $\Omega$,
\begin{equation} \label{eqn:An}
  (dd^c{U})^n\,
  :=\,\underbrace{dd^c{U}\wedge\dots \wedge dd^c{U}}_{n \ \text{factors}}
  = A_n (\det \cHess U ) \V_n, \quad \text{where } A_n:= 4^n n \, !\, .
\end{equation}
Suppose $U \in {\rm psh}(\Omega) \cap \smoo^2(\Omega)$, then $\det \cHess U \geq 0$. Therefore,
by \eqref{eqn:An},
$ (dd^c{U})^n$ is a positive $(n,n)$-current and hence is in $\mathcal{M}_{n}(\Omega)$.
Bedford--Taylor \cite[Proposition~2.3]{bedford-taylor:1976} showed that there exists
a unique continuous operator
$$
  (dd^c(\bcdot))^n: {\rm psh}(\Omega) \cap \smoo(\Omega) \to \mathcal{M}_{n}(\Omega)
$$
that extends $(dd^c(\bcdot))^n|_{{\rm psh}(\Omega) \cap \smoo^2(\Omega)}$ (which is
to be understood in the usual sense as in \eqref{eqn:An}).
As ${\rm psh}(\Omega) \cap \smoo^2(\Omega)$ is dense in
${\rm psh}(\Omega) \cap \smoo(\Omega)$, $(dd^c{u})^n$ is a positive
$(n,n)$-current for every $u \in {\rm psh}(\Omega) \cap \smoo(\Omega)$.
\smallskip

The complex Monge--Amp{\`e}re operator $(dd^c(\bcdot))^n$ can be defined as a
$(n,n)$-current even when $u$ lies in larger classes, for instance, when
$u \in {\rm psh}(\Omega) \cap \mathbb{L}_{\rm loc}^{\infty}(\Omega)$.
This relies on the fact that $dd^c u$ is a closed positive $(1,1)$-current
and on the Bedford--Taylor theory for defining ``products'' of such currents.
For more details about $\opCMA$, see, for instance, \cite{bedford-taylor:1976},
\cite{bedford-taylor:1982}, \cite{blocki:1996}, \cite{guedj-zeriahi:2017}.
\smallskip

\subsection{The complex Monge--Amp{\`e}re equation on $B$-regular domains}
\label{subsec:defn-prpn} 
We begin this section with a result of Sibony \cite{sibony:1987} that encodes the
definition of $B$-regular domains.
\begin{result}[paraphrasing of Sibony {\cite[Th{\'e}or{\`e}me~2.1]{sibony:1987}}, 
B{\l}ocki {\cite[Theorem~1.7]{blocki:1996}}]
\label{res:B-reg_equiv-prop}
Let $\Omega \varsubsetneq \Cn$ be a bounded domain. The following are equivalent.
\begin{itemize}
  \item[$(a)$] For each $\zeta \in \bdy \Omega$, $\Omega$ admits a strong plurisubharmonic barrier,
  i.e., there exists $v_{\zeta} \in \smoo(\overline{\Omega})\cap {\rm psh}(\Omega)$ such that 
  $v_{\zeta}(\zt)=0$ and $v_{\zt}|_{\Omegabar \setminus \{\zeta\} } < 0$.
  
  \item[$(b)$] $\Omega$ admits an exhaustion function $\psi \in \cont \cap \smoo^{\infty}(\Omega) \cap
  {\rm psh}(\Omega) $ with $\psi|_{\bdy \Omega}\equiv 0$ such that
  $$
    \sum_{j,k =1}^{n} \frac{\bdy^2 \psi}{\bdy z_j\bdy \overline{z}_k} v_j \overline{v}_k \geq \|v\|^2,
    \quad \forall v=(v_1, \cdots, v_n) \in \Cn.
  $$

  \item[$(c)$] For every function $f \in \smoo(\bdy\Omega; \R)$, there exists $v \in \cont \cap
  {\rm psh} (\Omega)$ such that $v|_{\bdy \Omega} = f$.
\end{itemize}
\end{result}

\begin{definition}\label{defn:B-reg}
A bounded domain $\Omega \varsubsetneq \Cn$ is said to be \emph{$B$-regular} if it satisfies
any of the equivalent conditions stated in Result~\ref{res:B-reg_equiv-prop}.
\end{definition}

\begin{remark}
Result~\ref{res:B-reg_equiv-prop} is attributed to two authors.
In \cite{sibony:1987}, Sibony introduced the notion of $B$-regularity.
In fact, $\bdy \Omega$ being $B$-regular, as defined in \cite{sibony:1987},
is equivalent to each of the conditions stated in
Result~\ref{res:B-reg_equiv-prop} under the assumption that $\bdy \Omega$ is
$\smoo^1$-smooth; see \cite[Th{\'e}or{\`e}me~2.1]{sibony:1987}. The equivalence of these
three conditions on general bounded domains was established by B{\l}ocki; see
\cite[Theorem~1.7]{blocki:1996}. We attribute Result~\ref{res:B-reg_equiv-prop} to Sibony as well
as to B{\l}ocki because, in addition to the preceding
observations, B{\l}ocki himself attributes \cite[Theorem~1.7]{blocki:1996} to Sibony.
\end{remark}
\smallskip

We shall now state the result by B\l{}ocki \cite{blocki:1996}, alluded to in
Section~\ref{sec:intro}, regarding the existence and uniqueness of solution to \eqref{eqn:D-P_CMA}
on $B$-regular domains. To do so, we first define the following Perron--Bremermann family:
given $\phi \in \smoo(\bdy \Omega; \R)$, $f \in \smoo(\overline{\Omega})$ with $f \geq 0$,
\begin{equation}\label{eqn:Blo-PB-envp}
\widehat{\mathscr{B}}(\phi, f) := \{ v \in \smoo(\overline{\Omega}) \cap {\rm psh}(\Omega):
v|_{\bdy \Omega}= \phi, \ (dd^c v)^n \geq f \V_n \}.
\end{equation}

\begin{result}[paraphrasing {\cite[Theorem~4.1]{blocki:1996}}] \label{res:B-reg_CMA-soln}
Let $\Omega \varsubsetneq \Cn$ be a $B$-regular domain. Let $\phi \in \smoo(\bdy \Omega; \R)$ and
$f \in \smoo(\overline{\Omega})$ with $ f \geq 0$. Then, there exists a unique solution
$u \in \smoo(\overline{\Omega}) \cap {\rm psh}(\Omega)$ to the Dirichlet problem
\eqref{eqn:D-P_CMA}, where this solution is obtained as follows:
$$
  u(z) = \sup \{ v(z): v \in \widehat{\mathscr{B}} (\phi, f) \}.
$$
\end{result}

The above result is strongly related to \cite[Theorem~8.3]{bedford-taylor:1976}.
We mention that certain other Perron--Bremermann families corresponding to continuous data
$(f, \phi)$\,---\,which are larger than the one defined by
\eqref{eqn:Blo-PB-envp}\,---\,are also considered in the literature to
describe the solution to \eqref{eqn:D-P_CMA}; see, for example,
\cite{bedford-taylor:1976, li:2004, ha-khanh:2015}. We note a couple of such families
that appear frequently:
\begin{align} \label{eqn:PB-envp}
  \mathscr{F}(\phi, f) &:= \{ v \in  \smoo(\overline{\Omega}) \cap {\rm psh}(\Omega) : v|_{\bdy \Omega} \leq \phi,
  \ (dd^c v)^n \geq f \V_n \},  \notag \\
  \mathscr{B}(\phi,f ) &:= \Big\{ v \in \smoo(\Omega) \cap {\rm psh}(\Omega): \limsup_{\Omega\ni z\to \zt} v(z)\leq
  \phi(\zt) \ \forall \zt\in \bdy\Omega \text{ and } (dd^c v)^n \geq f \V_n \Big\}.
\end{align}
The following proposition asserts that the upper envelopes of all Perron--Bremermann
families defined above coincide in the case of $B$-regular domains. We conclude this
section by stating and proving this proposition.
The result stated
below is a version of the comparison principle, which will be useful in the proof.
\begin{result}[see {\cite[Corollary~2.2.8]{blocki-notes:2002}}, for instance]
\label{res:comp-prin}
Let $\Omega$ be a bounded domain in $\Cn$. Let $u, v \in {\rm psh}(\Omega) \cap
\leb{\infty}(\Omega) $
be such that $\liminf_{z \to \zeta}(u(z) - v(z)) \geq 0$ for all $\zeta \in \bdy \Omega$.
Suppose $(dd^cu)^n \leq (dd^cv)^n$. Then, $u \geq v$ on $\Omega$.
\end{result}

\begin{proposition} \label{prpn:eq-envelopes}
Let $\Omega \varsubsetneq \Cn$ be a $B$-regular domain. Let $\phi \in \smoo(\bdy \Omega; \R)$
and $f \in \smoo(\overline{\Omega})$ with $ f \geq 0$. Then,
$$
  \sup \{ v(z): v \in \Bloclass \} = \sup \{ v(z): v \in \mathscr{F}(\phi, f) \}
  = \sup \{ v(z): v \in \mathscr{B}(\phi, f) \} \quad \forall z \in \Omega.
$$   
\end{proposition}

\begin{proof}
Since $\Bloclass \subset \mathscr{F}(\phi, f) \subset \mathscr{B}(\phi,f ) $, it suffices
to show that
$$
  \sup \{ v(z): v \in \Bloclass \} \geq \sup \{ v(z): v \in \mathscr{B}(\phi, f) \} \quad \forall z \in \Omega.
$$
In view of Result~\ref{res:B-reg_CMA-soln}, the function $u(z):= \sup \{v(z): v \in \Bloclass\}
\in \cont \cap {\rm psh}(\Omega) $ is the unique solution to \eqref{eqn:D-P_CMA}.
Let us write 
$$
  U(z):= \sup \{ v(z): v \in \mathscr{B}(\phi, f) \} \quad \forall z \in \Omega.
$$
We will now show that $U(z) \leq u(z)$ for all $z \in \Omega$. To this end, pick an arbitrary
$v \in \mathscr{B}(\phi, f)$.
\medskip

\noindent{\bf Claim.} There exists $v_1 \in \mathscr{B}(\phi, f) \cap
\mathbb{L}^{\infty}(\Omega)$ such that $v \leq v_1$. 
\smallskip

\noindent{\em Proof of claim:}
Since $\Omega$ is $B$-regular, by Result~\ref{res:B-reg_equiv-prop},
\begin{itemize}
  \item there exists $v_{\phi} \in \cont \cap {\rm psh}(\Omega) $ such that $v_{\phi}|_{\bdy \Omega} = \phi$, and
  \smallskip

  \item there exists an exhaustion function
  $\psi \in \cont \cap \smoo^{\infty}(\Omega) \cap {\rm psh}(\Omega)$
  with $\psi|_{\bdy \Omega}\equiv 0$ such that
  \begin{equation} \label{eqn:B-reg-smooth-exh}
    \sum_{j,k =1}^{n} \frac{\bdy^2 \psi}{\bdy z_j\bdy \overline{z}_k} v_j \overline{v}_k \geq \|v\|^2,
    \quad \forall v=(v_1, \cdots, v_n) \in \Cn.
   \end{equation}
\end{itemize}
Let us choose a constant $C>0$ such that $ C^n \geq {A_n}^{-1}\sup \{|f(z)|: z \in \Omegabar\} $,
where $A_n>0$ is the constant specified in \eqref{eqn:An}.
Now define
$$
  v_1: = \max \{ v, v_{\phi} + C \psi\} \in {\rm psh}(\Omega) \cap \smoo(\Omega).
$$
The next two inequalities, in the sense of currents,
are due to standard properties \cite[Proposition~2.8]{bedford-taylor:1976}
of $(dd^c(\bcdot))^n$. Note
\begin{align*}
  (dd^c(v_{\phi} + C \psi))^n &\geq (dd^c v_{\phi})^n + C^n (dd^c \psi)^n \\
  &\geq C^n A_n (\det \cHess \psi ) \V_n \geq C^n A_n \V_n \geq f \V_n,
\end{align*}
where the last two inequalities are due to \eqref{eqn:B-reg-smooth-exh} and the choice of
$C>0$ above. Moreover,
\begin{align*}
  (dd^cv_1)^n \geq \min \{ (dd^c v)^n, (dd^c(v_{\phi} + C \psi))^n \} \geq f \V_n.
\end{align*}
It is easy to see, from the definition of $v_1$, that 
$\limsup_{\Omega\ni z\to \zt}v_1(z)
\leq \phi(\zt)$ for every $\zt\in \bdy\Omega$,
$v_1 \in \mathbb{L}^\infty(\Omega)$, and $v \leq v_1$.
This proves the claim.
\hfill $\btl$
\smallskip

We now appeal to the comparison principle Result~\ref{res:comp-prin}. Recall that
$(dd^c u)^n= f \V_n$, $u|_{\bdy \Omega} = \phi$. By the above claim $(dd^c v_1)^n \geq f \V_n$,
and $ \liminf_{z \to \zt}(u(z) - v(z)) \geq 0$ for all $\zt \in \bdy \Omega$.
Thus, by Result~\ref{res:comp-prin} $u \geq v_1$ on $\Omega$. Hence, $u \geq v_1 \geq v$
on $\Omega$. Since $v \in  \mathscr{B}(\phi, f)$ was arbitrary, we get
$$
  U(z) \leq u(z) \quad \forall z \in \Omega.
$$
This completes the proof.
\end{proof}

\section{The proof of Theorem~\ref{th:general-reg-CMA}}\label{sec:general-reg-result}   
\smallskip

This section is devoted to proving Theorem~\ref{th:general-reg-CMA}. We begin with a
definition that was deferred from Section~\ref{sec:intro}. 

\begin{definition}\label{defn:mod-cont}
A function $\omega: [0,\infty)\to [0,\infty)$ is called a \emph{modulus of
continuity} if it is concave, monotone increasing, and such that
$\lim_{x\to 0^+}\omega(x) = \omega(0) = 0$.
\end{definition}
With $\omega$ as above, we define
\begin{equation}\label{eqn:omega-class}
  \smoo_{\omega}(\overline{\Omega}):= \{\psi\in \smoo(\overline{\Omega}): \exists C>0 \text{ such that }
  |\psi(x)- \psi (y)| \leq C \omega (\|x-y \|) \ \forall x, y \in \overline{\Omega}\}.
\end{equation}

We state a few results that will be useful in our proof.
\smallskip

The first of these results is a mixed Monge--Amp{\`e}re inequality. Its statement is a very mild
paraphrasing\,---\,done in order to be compatible with the notation in \eqref{eqn:D-P_CMA} and in
the proof of Theorem~\ref{th:general-reg-CMA}\,---\,of \cite[Theorem~1.2]{dinew:2009} by Dinew (also
refer to \cite{kolodziej:2003}).

\begin{result}[Ko{\l}odziej~\cite{kolodziej:2003}, {Dinew~\cite[Theorem~1.2]{dinew:2009}}]
\label{res:mixed_CMA}
Let $\Omega$ be a domain in $\Cn$ and let $v, w$ be bounded plurisubharmonic functions on
$\Omega$. Let $g, h \in \leb{1}(\Omega, \boldsymbol{\lambda})$ be non-negative functions
such that 
\begin{equation}\label{eqn:current_ineqs_1}
  (dd^c v)^n \geq g \V_n \quad \text{and} \quad (dd^c w)^n \geq h \V_n,
\end{equation}
where $\boldsymbol{\lambda}$ denotes the Lebesgue measure. Then,
\begin{equation}\label{eqn:current_ineqs_2}
  (dd^c v)^{n-k} \wedge (dd^c w)^k \geq g^{\frac{n-k}{n}} h^{\frac{k}{n}} \V_n \quad \text{for each } k=0, 1, \cdots, n.
\end{equation}
\end{result}

\begin{remark}\label{rem:mixed_CMA}
Two observations about Result~\ref{res:mixed_CMA} are needed. Firstly: this result is
attributed to two authors because, while it appears in \cite[Theorem~1.2]{dinew:2009} by
Dinew, he attributes it to Ko{\l}odziej \cite{kolodziej:2003}. Ko{\l}odziej's result is
established in a slightly different set-up with $v$ and $w$ being continuous, but the ideas
in his proof apply verbatim to Result~\ref{res:mixed_CMA}, which is explained
by Dinew right after the statement of \cite[Theorem~1.2]{dinew:2009}. Secondly:
\cite[Theorem~1.2]{dinew:2009} is a result featuring inequalities that involve measures that
are absolutely continuous with respect to $\boldsymbol{\lambda}$. These translate into
\eqref{eqn:current_ineqs_1} and \eqref{eqn:current_ineqs_2} owing to the definition of
$(dd^c(\bcdot))^k$ on ${\rm psh}(\Omega)\cap \leb{\infty}(\Omega, \boldsymbol{\lambda})$
and the fact that $\V_n$ is the standard volume form on $\R^{2n}$.
\end{remark}

\begin{lemma}\label{l:aux-psh}
Let $\Omega \varsubsetneq \Cn$ be a bounded domain and let $\psi \in \smoo^{1,1}(\overline{\Omega}; \R)$.
Suppose that $\Omega$ admits $\rho \in \cont \cap {\rm psh}(\Omega) $ such that
$ \Omega= {\rho}^{-1}(-\infty, 0)$,
$\bdy \Omega = \rho^{-1}\{0\}$, and $\rho -  \|\bcdot\|^2 \in {\rm psh}(\Omega)$.
Then, there exists $K > 0$ large enough such that $\psi + K \rho \in \cont \cap
{\rm psh} (\Omega)$.
\end{lemma}
The above lemma follows for the following reason: as $\psi\in \smoo^{1,1}(\overline{\Omega})$,
it follows from Rademacher's theorem that each second-order partial derivative is defined
at a.e. $z\in \overline{\Omega}$ and determines a function of class
$\mathbb{L}^{\infty}(\overline{\Omega})$. 
Therefore, given that $dd^c \rho \geq dd^c (\|\bcdot\|^2)$ in the sense of currents,
one can now choose $K>0$ large enough that $dd^c (\psi + K \rho)$
is a positive current. Hence the desired conclusion follows.
\smallskip

Before we prove Theorem~\ref{th:general-reg-CMA}, the following remark is in order:

\begin{remark} \label{rem:detailed-proof}
As discussed in Section~\ref{sec:intro}, Theorem~\ref{th:general-reg-CMA} is strongly motivated by 
\cite{li:2004}. In particular, our sufficient condition involving a function
$\rho \in \cont \cap {\rm psh}(\Omega) $ is inspired by \cite[Corollary~2.2]{li:2004}. Its influence
is seen in Steps~1 and~2 of our proof. Now, a careful reading of the proofs in \cite{li:2004} shows
that the proof of \cite[Corollary~2.2]{li:2004} relies on \cite[Proposition~2.1]{li:2004}. 
Our proof
diverges markedly from the argument for \cite[Proposition~2.1]{li:2004} at Step~3 of our
proof. Let $u$ denote the solution to the problem \eqref{eqn:D-P_CMA}. In proving 
\cite[Proposition~2.1]{li:2004}, $u$ is obtained as the upper-semicontinuous regularization
of the upper envelope of the \emph{truncated} Perron--Bremmerman family given in line~23
of page~122 of \cite{li:2004} (in contrast to our use of $\widehat{\mathscr{B}}(\phi, f)$). Showing
that this $u$ is not merely subsolution relies on adapting an argument in
\cite{bedford-taylor:1976}, which is hinted at in \cite{li:2004}. This involves solving
an auxiliary Dirichlet problem on the ball, following \cite[Section~8]{bedford-taylor:1976}, which is
an approach that becomes very technical for the general moduli of continuity that we consider.
Instead, starting at Step~3 below, we take a different, direct approach, which relies on
Result~\ref{res:mixed_CMA}.
\end{remark}

\begin{proof}[The proof of Theorem~\ref{th:general-reg-CMA}]
We fix a $\phi \in \smoo^{1,1}(\bdy \Omega; \R)$ and an $f \geq 0$ satisfying
$f^{1/n} \in \smoo_{\omega}(\overline{\Omega})$. As we recall from Section~\ref{subsec:defn-prpn},
\begin{equation} \label{eqn:Blo-env}
  \widehat{\mathscr{B}}(\phi, f) = \{ v \in \cont \cap {\rm psh}(\Omega):
  v|_{\bdy \Omega}= \phi, \ (dd^c v)^n \geq f \V_n \}.
\end{equation}
Let $u$ be the upper envelope of this Perron--Bremermann family, i.e.,
\begin{align} \label{eqn:soln-candidate}
  u(z) &= \sup \{ v(z): v \in \widehat{\mathscr{B}} (\phi, f) \}, \ z \in \overline{\Omega}.
\end{align}
We shall present the proof in the following four steps. In the first step, we explain why the
function $u$ defined in \eqref{eqn:soln-candidate} solves \eqref{eqn:D-P_CMA}. In the
subsequent steps, we will show that $u \in \omgclass$.%
\medskip

\noindent{\textbf{Step 1.}} \emph{Showing that $u$ solves the Dirichlet problem \eqref{eqn:D-P_CMA} }
\smallskip

\noindent{First,} we note that for each $\zeta \in \bdy \Omega$, $\Omega$ admits a
strong plurisubharmonic barrier $v_{\zeta}$. Define
$$
  v_{\zeta}(z):= - \|z - \zeta\|^2 + \rho(z), \quad \forall z \in \Omegabar.
$$
Clearly, given the assumptions on $\rho$, $v_{\zeta}$ is a strong plurisubharmonic barrier
for each $\zeta \in \bdy \Omega$. So, by Result~\ref{res:B-reg_equiv-prop},
$\Omega$ is a $B$-regular domain. Thus, by Result~\ref{res:B-reg_CMA-soln}, the function
$u \in \cont \cap {\rm psh}(\Omega)$ defined by \eqref{eqn:soln-candidate}
solves the Dirichlet problem \eqref{eqn:D-P_CMA} with data $(f, \phi)$ as fixed above.
\medskip

\noindent{\textbf{Step 2.}} \emph{Showing that there exists $ K_1 >0$ such that
$|u(z) - u (\zeta)| \leq K_1 \omega ( \|z - \zeta\|) $ for all
$(z, \zeta) \in \overline{\Omega} \times \bdy \Omega$}
\smallskip

\noindent{We} extend $\phi$ to a function in $\smoo^{1,1}(\overline{\Omega}; \R)$ and call
it $\phi$ as well. For $K>0$, let us write
$$
  v_K(z):= \phi(z) + K \rho (z) \quad \text{and} \quad \widetilde{v}_K(z):= - \phi(z) + K \rho (z).
$$
By Lemma~\ref{l:aux-psh}, there is $K>0$ such that both the functions $v_K$ and
$\widetilde{v}_K \in \cont \cap {\rm psh}(\Omega) $.
Let us {\bf fix} such a $K>0$.
\medskip

\noindent{\bf Claim.} We can choose $K^\bullet \geq K$ so large that $v_{K^\bullet}$,
$\widetilde{v}_{K^\bullet} \in \cont \cap {\rm psh}(\Omega)$, and
$(dd^cv_{K^\bullet})^n \geq f \V_n$.%
\smallskip

\noindent{{\em Proof of claim:}}
First, we note that for any $K^{\bullet}$ such that $K^{\bullet} \geq K$, both of the functions
$$
  v_{K^\bullet}:= \phi + K^\bullet \rho \quad \text{and}
  \quad \widetilde{v}_{K^\bullet}:= - \phi + K^\bullet \rho
$$
are in $ \cont \cap {\rm psh}(\Omega)$. We now appeal to the subadditivity of
$(dd^c(\bcdot))^n$ on ${\rm psh}(\Omega) \cap \smoo(\Omega)$\,---\,see 
\cite[Proposition~2.8-(2)]{bedford-taylor:1976}\,---\,to get the
following two inequalities, in the sense of currents. Note that
\begin{align*}
  (dd^c \rho)^n &= (dd^c (\rho - \|\bcdot\|^2 + \|\bcdot\|^2))^n \\
  &\geq (dd^c (\rho - \|\bcdot\|^2))^n +  (dd^c(\|\bcdot\|^2))^n
  \geq (dd^c(\|\bcdot\|^2))^n = A_n \V_n.
\end{align*}
Therefore,
\begin{align}\label{eqn:aux-subsoln-claim}
  (dd^c (\phi + K^\bullet \rho))^n &= (dd^c (\phi + K \rho + (K^\bullet - K) \rho))^n \notag \\
  &\geq (dd^c (\phi + K \rho ))^n + (K^\bullet - K)^n(dd^c \rho)^n \notag \\
  &\geq (K^\bullet - K)^n A_n \V_n.
\end{align}
Here $A_n = 4^n n \, !$ (as specified in \eqref{eqn:An}). We can now choose $K^\bullet>0$
so large that 
$$
  (K^\bullet - K)^n A_n \geq \sup \{|f(z)|: z \in \Omegabar \}.
$$
From the above and \eqref{eqn:aux-subsoln-claim} we get $(dd^cv_{K^\bullet})^n \geq f \V_n$,
in the sense of currents.
\hfill $\btl$
\smallskip

Let us fix a $K^\bullet >0$ as given by the above claim. For simplicity, we write
$$
  v:= v_{K^\bullet} = \phi + K^\bullet \rho \quad \text{and} \quad \widetilde{v}:=
  \widetilde{v}_{K^\bullet} = - \phi + K^\bullet \rho .
$$
By the above claim, $v \in \Bloclass$ and hence $v \leq u$. Also,
$u + \widetilde{v} \in \smoo(\overline{\Omega}) \cap {\rm psh}(\Omega)$ and
$(u + \widetilde{v})|_{\bdy \Omega}=0$. Therefore,
\begin{equation} \label{eqn:sub-sup-estimates}
  v \leq u \leq - \widetilde{v} \ \text{ on } \overline{\Omega} \quad \text{and} \quad
  u|_{\bdy \Omega}=v|_{\bdy \Omega}= -\widetilde{v}|_{\bdy \Omega} = \phi,
\end{equation}
where the second inequality holds because of the maximum principle for plurisubharmonic
functions. Therefore, by \eqref{eqn:sub-sup-estimates}, for $z \in \Omegabar$ and
$\zeta \in \bdy \Omega$,
\begin{align} \label{eqn:assymetry-estm}
  |u (z) - u (\zeta)| = \begin{cases} u(z) - u(\zeta) \leq - \widetilde{v}(z) +
  \widetilde{v}(\zeta) \leq |\widetilde{v}(z)- \widetilde{v}(\zeta) |,
  \quad &\text{if } u(z)\geq u (\zeta)\\
                                       u(\zeta) - u(z) \leq v(\zeta) - v(z) \leq |v(z) - v(\zeta)|,
                                       \quad &\text{if } u(\zt)\geq u (z)
                        \end{cases}
\end{align}
Now, as $\omega$ is concave and $\omega(0)= 0$, there exists $c>0$ such that 
$$
  \omega(x) \geq c x \quad \text{near } 0 \in \R.
$$
So, we have $\phi \in \omgclass$ and hence $v, \widetilde{v} \in \omgclass$.
Thus, in view of \eqref{eqn:assymetry-estm}, there exists $K_1>0$ such that
\begin{equation} \label{eqn:soln-aux-estimate}
  |u(z) - u (\zeta)| \leq K_1 \omega ( \|z - \zeta\|) \quad
  \forall (z, \zeta) \in \overline{\Omega} \times \bdy \Omega.
\end{equation}

\medskip

\noindent{\textbf{Step 3.}} \emph{Constructing auxiliary functions $V_{\tau}: \Omegabar \to \R$
for $\tau \in \Cn, \ \|\tau\|$ small, such that $V_{\tau} \in \Bloclass$}
\smallskip

\noindent{For} $\tau \in \Cn, \ \|\tau\|$ small enough, we define
\begin{equation} \label{eqn:vtau}
  v_{\tau}(z):= u (z + \tau) + C_f \omegatau \|z\|^2 - K_1 \omegatau - K' \omegatau \quad
  \forall z \in \overline{\Omega - \tau},
\end{equation}
where $C_f>0$ is a constant (that depends only on $f$) to be determined later
(see \eqref{eqn:const-choice}), $K_1>0$ is given by $\eqref{eqn:soln-aux-estimate}$, and $K'>0$ is so chosen
that $K' > C_f \sup_{z \in \Omegabar} \|z\|^2$.
\smallskip

Now, using \eqref{eqn:soln-aux-estimate} and \eqref{eqn:vtau}, we get
\begin{align}\label{eqn:vtau-estm}
  v_{\tau}(z) \leq u(z) + (C_f \|z\|^2 - K')\omegatau \leq u(z) \quad
  \forall z \in \bdy (\Omega - \tau) \cap \Omegabar,
\end{align}
where the second inequality is due to the choice of $K'$ above.
\smallskip

We now appeal to the following construction, which is inspired by Walsh
\cite{walsh:1968}, to complete the proof. Define for $\tau \in \Cn, \ \|\tau\|$ small enough,
$$
  V_{\tau}: \Omegabar \to \R, \  
  V_{\tau}(z):=
  \begin{cases}
   u(z) \quad  &\text{if } z \in \Omegabar \setminus (\Omega - \tau), \\
   \max \{u(z), v_{\tau}(z)\} \quad &\text{if } z \in \Omegabar \cap \overline{\Omega - \tau}.
  \end{cases}
$$
In view of \eqref{eqn:vtau-estm}, it is easy to see that $V_{\tau}$ is well-defined
and is in $\cont \cap {\rm psh}(\Omega)$. We claim that $V_{\tau} \in \Bloclass$.
\smallskip

Note, by \eqref{eqn:soln-aux-estimate}, we get $v_{\tau} (\zeta) \leq u (\zeta)$ for each
$\zeta \in \bdy \Omega \cap (\Omega - \tau)$. Therefore,
$$
  V_{\tau} (\zeta) = u(\zeta) = \phi(\zeta) \quad \forall \zeta \in \bdy \Omega,
$$
by the definition of $V_{\tau}$. Thus, to establish the above claim, it remains to show that
$(dd^c V_{\tau})^n \geq f \V_n$. In view of \cite[Proposition~2.8]{bedford-taylor:1976},
to establish the latter, it suffices to show that
\begin{equation} \label{eqn:vtau_MA-estm-final}
  (dd^c v_{\tau})^n \geq f \V_n \quad \text{on } \Omega \cap (\Omega - \tau) .
\end{equation}
In the remainder of this step, we shall establish \eqref{eqn:vtau_MA-estm-final}.
\smallskip

Note that on $\Omega \cap (\Omega - \tau)$
$$
  (dd^c v_{\tau})^n = (dd^c (u(\bcdot + \tau) + w_\tau) )^n, \quad \text{where }
  w_{\tau}(z) :=C_f \omegatau \|z\|^2.
$$
Recall that $A_n>0$ is the constant given by \eqref{eqn:An} so that
$(dd^c U)^n = A_n (\det \cHess{U}) \V_n$ for any $U \in \smoo^2(\Omega)$.
So,
$$
  (dd^c w_{\tau})^n = C_f^n (\omegatau)^n A_n \V_n.
$$
Also, we have on $\Omega \cap (\Omega - \tau)$
$$
(dd^c u(\bcdot + \tau))^n = f (\bcdot + \tau) \V_n. 
$$
We now appeal to the mixed Monge--Amp{\`e}re inequality stated in Result~\ref{res:mixed_CMA}.
By Result~\ref{res:mixed_CMA} we have on $\Omega \cap (\Omega - \tau)$
\begin{align} \label{eqn:mixed-CMA-mass}
  (dd^c u(\bcdot + \tau))^{n-k} &\wedge (dd^c w_{\tau})^k \notag \\
   &\geq {(C_f \omegatau)}^k A_n^\frac{k}{n} f^{\frac{n-k}{n}}(\bcdot + \tau) \V_n, \quad \text{for each } k = 0, 1, \cdots, n.
\end{align}
Next, we expand (in the sense of currents, see, for instance, \cite[Corollary~3.20]{guedj-zeriahi:2017})
\begin{align} \label{eqn:CMA-binomial-estm}
  (dd^c v_{\tau})^n &= (dd^c (u(\bcdot + \tau) + w_\tau) )^n \notag \\
                    &= \sum_{k=0}^n {n \choose k} (dd^c (u(\bcdot + \tau) )^{n-k} \wedge
                    (dd^c w_\tau)^k  \notag\\
                    &\geq \sum_{k=0}^n {n \choose k} {(C_f \omegatau)}^k A_n^\frac{k}{n}
                    f^{\frac{n-k}{n}}(\bcdot + \tau) \V_n, \quad \text{on } \Omega \cap (\Omega - \tau),
\end{align} 
where the last inequality is due to \eqref{eqn:mixed-CMA-mass}.
\smallskip

We shall now choose $C_f>0$, which was introduced in \eqref{eqn:vtau}.
Let $c_f >0$ be such that 
\begin{equation} \label{eqn:density-estm}
  |f^{1/n}(x) - f^{1/n}(y)| \leq c_f \, \omega (\|x-y\|) \quad \forall x, \, y \in \Omegabar.
\end{equation}
Now let us choose
\begin{equation} \label{eqn:const-choice}
  C_f> c_f A_n^{- 1/n} .  
\end{equation}
So, by \eqref{eqn:density-estm} and \eqref{eqn:const-choice}
\begin{align} \label{eqn:ptwise-estm_density}
f(z) &\leq  \sum_{k =0}^{n} {n \choose k} c^k_f \, \omegatau^k \,
f^{\frac{n-k}{n}}(z + \tau) \notag \\
&\leq  \sum_{k =0}^{n} {n \choose k} A_n^{\frac{k}{n}} C^k_f \, \omegatau^k \,
f^{\frac{n-k}{n}}(z + \tau) \quad \forall z \in \Omega \cap (\Omega - \tau).
\end{align}
Moreover, \eqref{eqn:ptwise-estm_density} also implies that on $\Omega \cap (\Omega - \tau)$
$$
  f \V_n \leq  \sum_{k =0}^{n} {n \choose k} A_n^{\frac{k}{n}} C^k_f \, \omegatau^k \,
  f^{\frac{n-k}{n}}(z + \tau) \V_n, \quad \text{in the sense of currents}.
$$
Therefore, from \eqref{eqn:CMA-binomial-estm} we get
$$
  (dd^c v_\tau)^n \geq f \V_n \quad \text{on } \Omega \cap (\Omega - \tau).
$$
This establishes \eqref{eqn:vtau_MA-estm-final}. The latter implies
$(dd^c V_{\tau})^n \geq f \V_n$, and hence $V_\tau \in \Bloclass$.

\medskip

\noindent{\textbf{Step 4.}} \emph{Showing that $u \in \omgclass$}
\smallskip

\noindent{Let} $z, z' \in \Omegabar$, $\|z-z'\|$ be small enough, and write $\tau:= z'-z$.
Now, $V_{\tau} \in \Bloclass$ by Step~3. Note that $z \in \Omegabar \cap \overline{\Omega - \tau}$.
So, by the definition of $u$ and $V_\tau$, we get
$$ 
  v_{\tau}(z) \leq V_{\tau}(z)\leq u(z) .
$$
Then, in view of $\eqref{eqn:vtau}$, the above inequality gives
\begin{align*}
  u(z') - u(z) &\leq (K_1+K' )\omegatau  - C_f \omegatau \|z\|^2 \leq (K_1+K' )\, \omega (\|z-z'\|).
\end{align*}
Reversing the roles of $z$ and $z'$ and proceeding exactly as above, we conclude that $u \in \omgclass$.
This completes the proof.
\end{proof}

\section{The proofs of Theorem~\ref{th:B-reg-CMA} and Corollary~\ref{cor:Holder-iff}}
\label{sec:two-proofs}

Before we prove Corollary~\ref{cor:Holder-iff}, we must mention that domains that
satisfy condition~$(a)$ in this result are abundant. Clearly, any bounded
strongly pseudoconvex domain will satisfy this condition. Moreover, any
smoothly bounded pseudoconvex domain of finite type in $\C^2$, or smoothly bounded convex
domain of finite type in $\C^n$, satisfies condition~$(a)$. This follows from
\cite[Theorem~A]{fornaess-sibony:1989} or \cite[Theorem~4.1]{fornaess-sibony:1989}, as
appropriate, combined with a construction presented in the \emph{proof}
of \cite[Proposition~5.1]{fornaess-sibony:1989}.
(Also see \cite[Remark~1]{li:2004}.) The form that the consequences of
\cite[Theorem~A, Theorem~4.1]{fornaess-sibony:1989} take allows one to add to the list
of domains in $\Cn$, $n\geq 3$, that satisfy condition~$(a)$. The last point will be discussed in a
forthcoming work.

\begin{proof}[The proof of Corollary~\ref{cor:Holder-iff}]
The implication $(a) \Rightarrow (b)$ follows from Result~\ref{res:Li-Holder-reg},
owing to the fact that its proof does not
require
the assumption that $\bdy \Omega$ is $\smoo^2$-smooth, as discussed in
Section~\ref{sec:intro}. (Alternatively, we may appeal to Theorem~\ref{th:general-reg-CMA},
taking $\omega (t):= t^{\varepsilon}, \ t \geq 0$.)
To complete the proof, we use the fact that
whenever $f \equiv  0$ or $f>0$,
$$
  f \in \smoo^{0,\,\varepsilon}(\overline{\Omega}) \Leftrightarrow
  f^{1/n} \in \smoo^{0,\,\varepsilon} (\overline{\Omega}).
$$

Conversely, assume that $(b)$ holds. Define
$\phi: \bdy \Omega \to \R$ as
$$
  \phi(z):= - \|z\|^2.
$$
Clearly $\phi \in \smoo^{1, 1}(\bdy \Omega)$. Then, by $(b)$, there exist
$\varepsilon \in (0,1]$, and $u_{\phi} \in \smoo^{0, \varepsilon}(\overline{\Omega}) \cap
{\rm psh}(\Omega) $ that solves the Dirichlet problem
\eqref{eqn:D-P_CMA} with $\phi$ as defined above, and with $f\equiv 0$ on
$\overline{\Omega}$. Now, define
\begin{equation} \label{eqn:psh-candidate}
  \rho (z): = u_{\phi}(z) +  \|z\|^2 \quad \forall z \in \overline{\Omega}.   
\end{equation}

Clearly, $\rho \in  \smoo^{0, \varepsilon}(\overline{\Omega}) \cap {\rm psh}(\Omega)$
and $ \rho - \|\bcdot\|^2  \in {\rm psh}(\Omega)$. Also, it follows from
\eqref{eqn:psh-candidate} that $\rho|_{\bdy \Omega}\equiv 0$. By the maximum principle,
$\rho < 0$ on $\Omega$. This proves $(a)$.
\end{proof}

Finally, we deduce the proof of Theorem~\ref{th:B-reg-CMA} from
Theorem~\ref{th:general-reg-CMA}. Before doing so, we discuss
whether H{\"o}lder regularity of solutions to \eqref{eqn:D-P_CMA} can be expected
for all admissible data on general $B$-regular domains. The answer is in the \textbf{negative}, as
shown by Bharali--Masanta
\cite{{bharali-masanta:2025}}. They constructed a class of $B$-regular domains
for which \eqref{eqn:D-P_CMA} cannot admit H{\"o}lder solutions
for arbitrary data as prescribed by \eqref{eqn:D-P_CMA}, even for $\phi$ highly
regular. But their arguments suggest even more.
The following statement is a combination of \cite[Theorem~1.1]{bharali-masanta:2025} and
\cite[Example~2.4]{bharali-masanta:2025}.

\begin{result}[Bharali--Masanta, \cite{bharali-masanta:2025}] \label{res:non-Holder_CMA}
There exist $B$-regular domains $\Omega \varsubsetneq \Cn, n \geq 2$, with $\smoo^2$-smooth
boundary having the following property: there exist functions $\phi: \bdy \Omega \to \R$ that are
restrictions of $\smoo^{\infty}$-smooth functions defined on neighbourhoods of $\bdy \Omega$ such that,
for any $f \in \smoo(\Omegabar), \, f \geq 0$, the unique solution to the Dirichlet problem
\eqref{eqn:D-P_CMA} does not belong to $\smoo^{0, \varepsilon}(\Omegabar)$
for any $\varepsilon \in (0, 1]$.  
\end{result}

In fact, the \emph{proof} of \cite[Theorem~1.1]{bharali-masanta:2025} reveals that,
given any concave modulus of continuity $\widetilde{\omega}$, there exists a class of $B$-regular
domains with $\smoo^2$-smooth boundary
for which \eqref{eqn:D-P_CMA} cannot admit solutions in $\smoo_{\widetilde{\omega}}(\Omegabar)$
for arbitrary data $(f,\phi)$, even for $\phi$ highly regular. It is this that motivates the question
raised in Section~\ref{sec:intro}.
Theorem~\ref{th:B-reg-CMA} can be viewed as a counterpoint to the above discussion.

\begin{proof}[The proof of Theorem~\ref{th:B-reg-CMA}]
As $\Omega$ is $B$-regular, by Result~\ref{res:B-reg_equiv-prop},
it admits an exhaustion function $\rho \in \smoo(\overline{\Omega}) \cap \smoo^{\infty}(\Omega)
\cap {\rm psh}(\Omega)$ with $\bdy \Omega = \rho^{-1}\{0\}$ such that
$$
  \sum_{j,k =1}^{n} \frac{\bdy^2 \rho}{\bdy z_j\bdy \overline{z}_k} v_j \overline{v}_k \geq \|v\|^2,
  \quad \forall v=(v_1, \cdots, v_n) \in \Cn.
$$
Clearly, $\rho - \|\bcdot\|^2 \in {\rm psh}(\Omega)$. Also,
by the maximum principle $ \Omega= {\rho}^{-1}(-\infty, 0)$.
\smallskip

Let us define $\omega_{\rho}:[0, \infty) \to [0, \infty) $ as
$$
  \omega_{\rho} (r):= \sup \big\{|\rho(x) - \rho(y)| : x, y \in \overline{\Omega}, \ \|x-y\| \leq r \big\} \quad \forall r \geq 0.
$$
Note, $\omega_{\rho}$ is monotone increasing and $\omega_{\rho}(r) = M$ for all
$ r \geq {\rm diam}(\overline{\Omega})$, where
$$
  M:=\sup \big\{|\rho(x) - \rho (y)|: {x, y \in \overline{\Omega}} \big\}.
$$
In particular, $\sup_{r \in [0, \infty)}\omega_{\rho}(r) \leq M$. Also, as $\rho $ is
uniformly continuous on $\Omegabar$, $\omega_{\rho}$ is continuous at 0.
\smallskip

Let $\overline{\omega}_{\rho}$ be the least concave majorant of $\omega_{\rho}$, i.e.,
\begin{equation} \label{eqn:concave-majorant}
  \overline{\omega}_{\rho}(r):=\begin{cases}
                                  \inf_{l \in L(\omega_{\rho})} l(r), &\text{if } r > 0,\\
                                  0, &\text{if } r=0, 
                               \end{cases}  
\end{equation}
where $L(\omega_{\rho}):= \{l \in {\rm Aff}(\R) : l(r) \geq \omega_{\rho}(r) \ \forall r \geq 0
\}$, and ${\rm Aff}(\R)$ denotes the class of $\R$-affine maps from $\R$ to $\R$; see
\cite[Chapter~2, Section ~6]{DeVore-Lorentz:1993}. Note that $L(\omega_{\rho}) \neq \emptyset$ because,
from the discussion in the last paragraph, the constant function $r \mapsto M, \ 
\forall r \in \R$, belongs to $L(\omega_{\rho})$. Clearly,
$$
  0 \leq \omega_{\rho}(r) \leq \overline{\omega}_{\rho}(r) \leq M \quad \forall r \geq 0.
$$
So, $\overline{\omega}_{\rho}$ is finite valued. It follows immediately from the definition
\eqref{eqn:concave-majorant} that $\overline{\omega}_{\rho}|_{(0, \infty)}$ is concave.
Also, $\overline{\omega}_{\rho}$ is increasing. This follows, for
instance, from the fact that as $\omega_{\rho}$ is increasing, each function
$l \in L(\omega_{\rho})$ must be increasing.
\smallskip

We now show that $\lim_{r \to 0^+} \overline{\omega}_{\rho}(r)=0$. If not, then there
there exist $\eps>0$ and a sequence $\{ r_{\nu}\} \subset [0, \infty)$ with
$\lim_{\nu \to \infty} r_{\nu} =0$ such that $\overline{\omega}_{\rho}(r_{\nu}) \geq \eps$
for each $\nu$. This implies that
\begin{equation*}
 l(r_{\nu}) \geq \eps \quad \forall \nu \in \Z_+ , \ \forall l \in  L(\omega_{\rho}).
\end{equation*}
In particular,
\begin{equation} \label{eqn:limit-at-0}
 l(0) = \lim_{\nu \to \infty} l(r_\nu) \geq \eps \quad \forall l \in  L(\omega_{\rho}).
\end{equation}
Let us now construct $l_\eps \in  L(\omega_{\rho})$ as follows. As $\omega_{\rho}$ is continuous at $0$,
there exists $\delta>0$ such that $\omega_{\rho}(r) < \eps/2$ for all $r \in [0, \delta]$.
Define
$$
  l_\eps(r):= (\eps/2) + (M/\delta)r, \quad r \in \R.
$$
\smallskip
Clearly $l_\eps \in  L(\omega_{\rho})$, and $l_\eps(0) = \eps/2 \ < \eps$. This contradicts
\eqref{eqn:limit-at-0}. Thus we have $\lim_{r \to 0^+} \overline{\omega}_{\rho}(r) =
\overline{\omega}_{\rho} (0)=0$. Since $\overline{\omega}_{\rho}|_{(0, \infty)}$ is concave,
$\overline{\omega}_{\rho}$ is continuous at each (interior) point in
$(0, \infty)$. Hence $\overline{\omega}_{\rho}$ is continuous on $[0, \infty)$. So, by
continuity, $\overline{\omega}_{\rho}$ is concave on $[0, \infty)$ as well.
\smallskip

Let us write $\omega:= \overline{\omega}_{\rho} $. In view of the last two paragraphs,
$\omega$ is a (concave) modulus of continuity (in the sense of Definition~\ref{defn:mod-cont}).
Furthermore, since $\omega_{\rho} \leq \omega$, $\rho \in \smoo_{\omega}(\overline{\Omega})$.
The result now follows immediately from Theorem~\ref{th:general-reg-CMA}.
\end{proof}

\section*{Acknowledgements}
I am grateful to Prof. Gautam Bharali for introducing me to this
subject and for suggesting one of the problems in this paper.
I also thank him for many valuable discussions and for his help in improving the
exposition of this article. I am grateful to the anonymous referee for 
pointing out an error in the proof of Theorem~\ref{th:general-reg-CMA}
in the previous version of the manuscript. I also thank them for their suggestions
to reorganize the paper and to improve its exposition.
This work is supported by a postdoctoral fellowship from Tata Institute
of Fundamental Research, Centre for Applicable Mathematics.

\end{document}